\documentclass{amsart}

\usepackage{amsmath,amssymb,amsthm}
\usepackage{pinlabel}

\newtheorem{prop}{Proposition}[section]
\newtheorem{thm}[prop]{Theorem}
\newtheorem{lem}[prop]{Lemma}
\newtheorem{cor}[prop]{Corollary}

\theoremstyle{definition}

\newtheorem*{defn}{Definition}
\newtheorem*{exs}{Examples}
\newtheorem{rem}[prop]{Remark}

\newtheorem*{ack}{Acknowledgements}

\def\co{\colon\thinspace}

\newcommand{\C}{\mathbb C}
\newcommand{\CP}{\mathbb{C}\mathrm{P}}

\newcommand{\rmd}{\mathrm d}

\newcommand{\rme}{\mathrm e}

\newcommand{\F}{\mathbb F}

\newcommand{\MM}{\mathcal M}
\newcommand{\wtMM}{\widetilde{\mathcal M}}

\newcommand{\N}{\mathbb N}

\newcommand{\R}{\mathbb R}

\newcommand{\wtW}{\widetilde{W}}

\newcommand{\bfx}{\mathbf x}

\newcommand{\bfy}{\mathbf y}

\newcommand{\bfz}{\mathbf z}

\newcommand{\lra}{\longrightarrow}
\newcommand{\ra}{\rightarrow}

\DeclareMathOperator{\Aut}{Aut}

\DeclareMathOperator{\diam}{diam}
\DeclareMathOperator{\dR}{dR}

\DeclareMathOperator{\ev}{ev}

\DeclareMathOperator{\FS}{FS}

\DeclareMathOperator{\id}{id}
\DeclareMathOperator{\Int}{Int}

\DeclareMathOperator{\st}{st}

\newcommand{\xist}{\xi_{\st}}

\begin{document}

\author{Hansj\"org Geiges}
\author{Kai Zehmisch}
\address{Mathematisches Institut, Universit\"at zu K\"oln,
Weyertal 86--90, 50931 K\"oln, Germany}
\email{geiges@math.uni-koeln.de, kai.zehmisch@math.uni-koeln.de}

\title[The strong Weinstein conjecture]{Symplectic cobordisms and
the strong Weinstein~conjecture}

\date{}

\begin{abstract}
We study holomorphic spheres
in certain symplectic cobordisms and derive information
about periodic Reeb orbits in the concave end of these
cobordisms from the non-compactness of the relevant
moduli spaces. We use this to confirm the strong Weinstein conjecture
(predicting the existence of null-homologous Reeb links)
for various higher-dimen\-sio\-nal contact manifolds,
including contact type hypersurfaces in subcritical Stein
manifolds and in some cotangent bundles. The quantitative character
of this result leads to the definition of a symplectic capacity.
\end{abstract}


\maketitle


\section{Introduction\label{intro}}
In a previous paper~\cite{geze} we studied moduli spaces
of holomorphic discs in certain $4$-dimensional
symplectic cobordisms. As in the classical work of
Hofer~\cite{hofe93} on the Weinstein conjecture in dimension~$3$,
the non-compactness of these moduli spaces was used
to detect periodic Reeb orbits. Moreover, this
approach enabled us to give a unified view of many results in
$3$-dimensional contact topology and $4$-dimensional symplectic
topology.

In the present paper we extend this work to higher
dimensions. Using an idea that can be
traced back to McDuff~\cite{mcdu91}, we modify our set-up by
constructing a symplectic cap on the
convex end of the symplectic cobordism. We can then work with
moduli spaces of holomorphic spheres rather than discs,
which allows us to invoke a compactness theorem from symplectic field
theory~\cite{behwz03}.

Our main technical result (Theorem~\ref{thm:main}) makes
quantitative predictions about periodic Reeb orbits
in the concave end of the symplectic cobordisms under consideration.
In Corollary~\ref{cor:main} we rephrase this as
a statement about the Weinstein conjecture~\cite{wein79}
in the strong version proposed by Abbas et al.~\cite{ach05},
which will be recalled in Section~\ref{section:Weinstein}.
We also recover a result of McDuff about a class of
symplectic fillings whose boundary is necessarily connected
(Theorem~\ref{thm:McDuff}).

In Section~\ref{section:applications} we explore various
applications of these results. Our main focus is on
specific instances of the strong Weinstein conjecture
and on quantitative Reeb dynamics. These examples include
contact type hypersurfaces in subcritical Stein manifolds
(Corollary~\ref{cor:Stein}) and in cotangent bundles
over split manifolds $Q\times S^1$ (Corollary~\ref{cor:cotangent}).
As in our previous paper, the quantitative results
give rise to the definition of a symplectic capacity via the periods of
Reeb orbits on contact type hypersurfaces.

A typical application is the Weinstein conjecture for subcritically
Stein fillable contact manifolds (Corollary~\ref{cor:filling}).
This result is complementary to that of Albers--Hofer~\cite{alho09}
on the Weinstein conjecture for higher-dimensional
contact manifolds that are overtwisted in the sense
of Niederkr\"uger~\cite{nied06} and hence, as shown there,
do not admit semipositive
strong symplectic fillings. Albers--Hofer study holomorphic discs
in trivial symplectic cobordisms only, but it seems likely
that this can be extended to the more general cobordisms considered here.
For other recent work on the higher-dimensional Weinstein 
conjecture see~\cite{nire}.

The final two sections contain the proof of the main technical result. In
Section~\ref{section:completing} we describe a completion
of the symplectic cobordism. In Section~\ref{section:moduli}
we introduce certain moduli spaces of holomorphic spheres
in this completed cobordism; the non-compactness of these
moduli spaces implies the main result.

Part of the motivation for this paper comes from the recent article by
Oancea--Viterbo~\cite{oavi} on the topology of symplectic fillings.
At the end of the present paper we briefly indicate
the relation of our results with their work.
\section{The strong Weinstein conjecture}
\label{section:Weinstein}
Let $M$ be a closed $(2n-1)$-dimensional
manifold carrying a (cooriented) contact structure~$\xi$,
i.e.\ a tangent hyperplane field defined as $\xi=\ker\alpha$
for some $1$-form $\alpha$ such that
$\alpha\wedge (\rmd\alpha)^{n-1}$ is a volume form on~$M$.
We equip $M$ with the orientation induced by this
volume form; write $\overline{M}$ for $M$ with the reversed
orientation. The Reeb vector field $R=R_{\alpha}$ of $\alpha$
is defined by the equations $i_R\rmd\alpha=0$ and $\alpha(R)=1$.

Paraphrasing Weinstein~\cite{wein79} we say that
$(M,\xi)$ {\em satisfies the Weinstein conjecture} if for
{\em every\/} contact form $\alpha$ defining $\xi$ the corresponding
Reeb vector field $R_{\alpha}$ has a closed orbit.

When we speak of a contractible periodic orbit, the period
is not required to be the minimal one.

\begin{defn}
A {\bf Reeb link} for a contact form $\alpha$ is
a collection of periodic orbits of~$R_{\alpha}$, not necessarily of
minimal period. Its {\bf total action} is the sum of the periods.
\end{defn}

In other words, when we speak of a Reeb link we allow
the components of this link to be multiply covered. This convention is
important for the following definition.

\begin{defn}
Following Abbas et al.~\cite{ach05} we say that $(M,\xi)$
{\em satisfies the \textbf{strong} Weinstein conjecture\/}
if for every $\alpha$ defining $\xi$ there exists a {\em nullhomologous\/}
Reeb link.
\end{defn}

Here is a simple example of a whole class of contact manifolds
$(M,\alpha)$ admitting nullhomologous Reeb links. Consider a closed
manifold $B$ with a symplectic form~$\omega$ such that
$\omega/2\pi$ represents an integral cohomology class.
Let $\pi\co M\rightarrow B$ be a principal circle bundle
of (real) Euler class $e=-[\omega/2\pi ]$. Then the
connection $1$-form $\alpha$ on this bundle with curvature form~$\omega$,
i.e.\ $\rmd\alpha=\pi^*\omega$, is a contact form on $M$ whose
Reeb orbits are the fibres of the $S^1$-bundle.
When $B$ is a surface, the Euler class may be regarded as an integer,
and the $|e|$-fold multiple of the fibre is nullhomologous in~$M$.
Unless $B$ is a $2$-sphere, no multiple of the fibre will be
nullhomotopic. This follows from the
homotopy exact sequence
\[ \ldots\lra\pi_2(B)\lra\pi_1(S^1)\lra\pi_1(M)\lra\ldots\] 
of the fibration, since surfaces of genus at least one are aspherical.
In the case $\dim B\geq 4$, choose a $2$-dimensional integral homology class
of $M$ on which $e$ evaluates non-trivially. Any such class
can be represented by an embedded surface~$\Sigma$,
see~\cite[Th\'eor\`eme~II.27]{thom54}. Then the $|e(\Sigma)|$-fold
multiple of the fibre will be homologically trivial in $\pi^{-1}(\Sigma)$
and, {\em a fortiori}, in~$M$.
\section{The main technical result}
\label{section:main}
Our main theorem answers the strong Weinstein conjecture in
the affirmative for contact manifolds $(M,\alpha)$ that arise as
the strong concave boundary of a suitable compact symplectic cobordism
$(W,\omega)$ of dimension~$2n$, which we equip with the orientation
given by the volume form~$\omega^{n}$. Throughout this paper it is
understood that $n\geq 2$.

The convex end of $(W,\omega)$ is required to contain
one connected component (or collection of components)
$S$ of the type we describe next.
This class of manifolds includes spheres and ellipsoids with their
standard contact form.
\subsection{The boundary component $S$}
\label{subsection:S}
Let $(P,\omega_P)$ be a compact, connected $(2n-2)$-dimensional symplectic
manifold (with boundary) admitting a strictly plurisubharmonic potential,
by which we mean the following. We require the existence of
an almost complex structure $J_P$ tamed by $\omega_P$, i.e.\
$\omega_P(X,J_PX)>0$ for all non-zero tangent vectors~$X$,
and a smooth function $\psi_P\co P\rightarrow\R$ having the boundary
$\partial P$ as a regular level set and with
\[ \omega_P=-\rmd(\rmd\psi_P\circ J_P).\]

A straightforward calculation, cf.\ \cite[Lemma~4.11.3]{geig08},
shows that the restriction of a (strictly) plurisubharmonic function
to a non-singular holomorphic curve is (strictly) subharmonic,
whence the name. This entails the maximum principle
for such curves (even in the non-strict case).

Given any point of an almost complex manifold and a holomorphic
tangent vector at that point, one can find a local holomorphic curve
passing through that point in the given tangent
direction~\cite[Theorem~III]{niwo63}. Combined with the maximum
principle this implies that $\psi_P$ attains its maximum
on the boundary $\partial P$ only. After changing $\psi_P$ by
an additive constant we may assume $\min\psi_P=0$.

Equip the product $C:=P\times\CP^1$ with the almost complex structure
$J_C:=J_P\oplus i$ and the symplectic
form $\omega_C:=\omega_P+\omega_{\FS}$, where $\omega_{\FS}$
denotes the Fubini--Study form of total integral $\pi$ on $\CP^1$.
Observe that the K\"ahler manifold $(\CP^1,\omega_{\FS})$
may be interpreted as the one-point compactification of the
open unit disc $B_1$ in $\C$ with its standard area form.
On $P\times B_1$ we have the corresponding strictly plurisubharmonic function
$\psi:=\psi_P+|z|^2/4$. We write $\CP^1=B_1\cup\{\infty\}$ and
$P_{\infty}:=P\times\{\infty\}\subset C$.

The manifold $S$ is supposed to be a regular level set $\psi^{-1}(c)$
with $c$ smaller than both $\max\psi_P$ and $1/4$, or a collection
of connected components of such a level set. This choice of $c$
ensures that $S$ is a compact hypersurface in $\Int (P)\times B_1$.
It inherits the contact form
\[ \alpha_S:=-\rmd\psi\circ J_C|_{TS}; \]
the contact structure $\ker\alpha_S$ is given by the
$J_C$-invariant sub-bundle of the tangent bundle $TS$.
The Liouville vector field $Y$ for $\omega_C$ defined by
$i_Y\omega_C=-\rmd\psi\circ J_C$ satisfies $\rmd\psi (Y)>0$, so
the contact manifold $(S,\alpha_S)$ is the strong convex boundary
of the symplectic manifold $\bigl( \psi^{-1}([0,c]),\omega_C\bigr)$.
\subsection{The symplectic cobordism $(W,\omega)$}
\label{subsection:cobordism}
Here are the properties we require of the symplectic cobordism $(W,\omega)$:

\begin{itemize}
\item[(C1)] $(W,\omega)$ is compact, connected and $\pi$-semipositive
(see below for the definition).
\item[(C2)] The oriented boundary of $W$ equals
\[ \partial W=\overline{M}\sqcup M_+\sqcup S,\]
where $M_+$ is allowed to be empty. The manifolds $M$, $M_+$ and $S$
are not required to be connected.
\item[(C3)] $(M,\alpha)$ is the strong concave boundary
of~$(W,\omega)$. By this we mean that there is a Liouville vector field
$Y$ for $\omega$ defined near $M\subset W$
and pointing into $W$ along $M$ such that $\alpha=i_Y\omega|_{TM}$.
\item[(C4)] On a neighbourhood of $M_+\subset W$ there is an
$\omega$-tame almost complex structure $J_+$ relative to which
the boundary $M_+$ is $J_+$-convex, i.e.\ the $J_+$-invariant
subbundle of $TM_+$ is a contact structure.
\item[(C5)] $(S,\alpha_S)$ is a strong convex boundary component
of $(W,\omega)$.
\end{itemize}

For the definition of $\pi$-semipositivity, which is essentially
the one given in
\cite[Definition~6.4.5]{mcsa04}, recall that
a homology class $A\in H_2(W)$ is said to be {\em spherical\/} if it lies
in the image of the Hurewicz homomorphism $\pi_2(W)\rightarrow H_2(W)$.
Write $c_1$ for the first Chern class of the symplectic manifold
$(W,\omega)$, defined via any almost complex structure on $W$
in the contractible space of $\omega$-tame structures, and
$c_1(A)$ for the evaluation of this class on~$A$. By $\omega (A)$
we denote the evaluation of the de~Rham cohomology class
$[\omega]\in H^2_{\dR}(W)$ on~$A$.

\begin{defn}
Let $\kappa$ be a positive real number.
A symplectic manifold $(W,\omega)$ of dimension $2n$
is {\bf $\kappa$-semipositive} if any spherical class $A\in H_2(W)$
with $0<\omega (A)<\kappa$ and $c_1(A)\geq 3-n$ satisfies
$c_1(A)\geq 0$.
\end{defn} 

Note that this condition is automatically satisfied for
symplectic manifolds of dimension at most~$6$.

One particular case of interest to us is the one where, in addition to
conditions (C1) to (C5), the symplectic form $\omega$ is an
exact form $\omega=\rmd\lambda$ with $\lambda|_{TM}=\alpha$.
We shall refer to this as the {\em Liouville case},
since the conditions are those of a concave boundary in a Liouville
cobordism, cf.~\cite{geze}. In this case, too, $\kappa$-semipositivity
is automatic.
\subsection{The main theorem}
\label{subsection:main}
We can now formulate our main result.

\begin{thm}
\label{thm:main}
Given a symplectic cobordism $(W,\omega)$
satisfying conditions {\rm (C1)} to {\rm (C5)}, there exists a
nullhomologous Reeb link in its concave end $(M,\alpha )$ of total
action smaller than~$\pi$. In the Liouville case there
is in fact a contractible Reeb orbit of period smaller than~$\pi$.
\end{thm}

\begin{rem}
The choice $\kappa=\pi$ in (C1) is made purely for notational
convenience. In the general case, one would have to
replace the Fubini--Study form $\omega_{\FS}$ by $(\kappa/\pi)\omega_{\FS}$,
and the open unit disc $B_1$ by a disc of radius $\sqrt{\kappa/\pi}$.
The action of the Reeb link predicted by our theorem would then be
smaller than~$\kappa$.
\end{rem}

A neighbourhood of $M\subset (W,\omega)$ looks like a
neighbourhood of $\{ 0\}\times M$ in the half-symplectisation
$\bigl([0,\infty)\times M,\rmd(\rme^s\alpha)\bigr)$. This allows us
to define a symplectic form $\omega_-$ on the manifold
\[ (-\infty,0]\times M\cup_M W \]
by
\[ \omega_-:=\begin{cases}
                  \omega          & \text{on $W$},\\
                  \rmd(\rme^s\alpha) & \text{on $(-\infty,0]\times M$}.
                  \end{cases}
\]

Any contact form defining the cooriented contact structure
$\xi:=\ker\alpha$ can be written as $\rme^h\alpha$ for some smooth function
$h\co M\rightarrow\R$. Rescaling this by a constant function
(which does not change the Reeb dynamics qualitatively) we may
assume that $h$ takes negative values only. Replacing $M$ by
\[ \{ (h(x),x)\in (-\infty,0]\times M\co x\in M\}
\subset (-\infty,0]\times M\cup_M W  \]
we obtain a cobordism as in Theorem~\ref{thm:main},
with concave boundary $(M,\rme^h\alpha)$.

In the Liouville case we can define the collar of $M$ in $W$
via the Liouville vector field $Y\equiv\partial_s$ given by
$i_Y\omega=\lambda$. Then both $\lambda$ and $\rme^s\alpha$
are $Y$-invariant and evaluate to zero on~$Y$; since they
coincide on $TM$, they coincide near~$M$. In other words,
$\rme^s\alpha$ defines an extension of the primitive $\lambda$
to $(-\infty,0]\times M$.

These considerations lead to the following corollary.

\begin{cor}
\label{cor:main}
Given a symplectic cobordism $(W,\omega)$
satisfying conditions {\rm (C1)} to {\rm (C5)},
the contact manifold $(M,\xi)$ satisfies the strong Weinstein
conjecture. In the Liouville case, any contact form
defining $\xi$ has a contractible periodic
Reeb orbit. \qed
\end{cor}

\begin{exs}
(1) Let $(M,\xi)$ be a closed contact manifold of dimension $3$ or $5$
occurring as the concave end of a strong symplectic cobordism
whose convex end is $S^3$ or $S^5$, respectively,
with its standard contact structure~$\xist$. Then $(M,\xi)$ satisfies
the strong Weinstein conjecture.

(2) If $(M,\xi)$ with $\dim M=2n-1$ is Liouville cobordant to
$(S^{2n-1},\xist)$, e.g.\ if $(S^{2n-1},\xist)$ can be obtained from
$(M,\xi)$ by contact surgery, then any contact form defining $\xi$
has a contractible closed Reeb orbit.
\end{exs}

Related cobordism-theoretic arguments allow one to reprove the
result of Abbas et al.~\cite{ach05} that planar contact structures
on $3$-manifolds satisfy the strong Weinstein conjecture.
Our proof of Theorem~\ref{thm:main} rests on the
construction of a symplectic cap, see Section~\ref{subsection:cap}
below. A cap in the case of planar contact structures
has been constructed by Etnyre~\cite{etny04}. Details of the
ensuing holomorphic curves argument can be found in~\cite{doer11}.

The methods for proving Theorem~\ref{thm:main} also yield the following
result. On the face of it, this is stronger than
a result of McDuff~\cite[Theorem~1.4]{mcdu91}, but her proof
actually yields the result we formulate here. In fact, the proof
we give in Section~\ref{subsection:McDuff} below is essentially
hers, except that we paraphrase it in the more sophisticated
language of~\cite{mcsa04}.

\begin{thm}[McDuff]
\label{thm:McDuff}
Let $(W,\omega)$ be a symplectic cobordism satisfying conditions {\rm (C1)}
to {\rm (C5)}, but now with the additional assumption that
$M$ be empty, i.e.\ there is no concave boundary component.
Then $M_+$ is likewise empty.
\end{thm}

\section{Applications}
\label{section:applications}
Before we turn to the proof of Theorem~\ref{thm:main}, we describe
a number of applications. Some are parallel to the $4$-dimensional
applications of the `ball theorem' proved in~\cite{geze}; we shall be brief
in the discussion of those.
\subsection{Reeb dynamics}
In \cite{vite87} Viterbo proved the existence of closed characteristics
on compact contact type hypersurfaces in $\R^{2n}$ with its standard
symplectic structure. In \cite[Theorem~4.4]{vite99} he extended
this to contact type hypersurfaces in subcritical Stein manifolds;
an alternative proof was given by
Frauenfelder--Schlenk~\cite[Corollary~3]{frsc07}.
Our main theorem allows us to prove the strong Weinstein conjecture
in these situations. First we recall the basic definitions.

\begin{defn}
A hypersurface $M$ in a symplectic manifold $(V,\omega)$
is said to be of {\bf contact type}
if there is a Liouville vector field $Y$ for $\omega$ defined near and
transverse to~$M$. The hypersurface is said to be of
{\bf restricted contact type}
if $Y$ is defined on all of~$V$.
\end{defn}

It will be understood that a contact type hypersurface $M$ is equipped
with the contact form $i_Y\omega|_{TM}$. If $Y$ and hence the primitive
$i_Y\omega$ for $\omega$ is globally defined, this places us
in the Liouville case of our main theorem. Recall our convention from
Section~\ref{section:Weinstein}: when we say that $M$ satisfies the
Weinstein conjecture, we mean that {\em every\/} contact form defining
the contact structure $\ker (i_Y\omega|_{TM})$ has a closed Reeb orbit.
Periodic Reeb orbits of the specific contact form $i_Y\omega|_{TM}$
will be referred to as {\bf closed characteristics}.

\begin{defn}
A {\bf Stein manifold} is a complex manifold $(V,J)$
admitting a proper holomorphic embedding into some $\C^N$.
Then $(V,J)$ admits an exhausting (i.e.\ bounded from below and proper)
strictly plurisubharmonic function $\psi$, e.g.\ the restriction of
the function $\sum_{k=1}^N |z_k|^2$ on~$\C^N$. The $2$-form
\[ \omega_{\psi}:=-\rmd(\rmd\psi\circ J)\]
is then a symplectic form on~$V$. By \cite[Theorem~1.4.A]{grel91},
any other exhausting strictly plurisubharmonic function on $(V,J)$
gives rise to a symplectomorphic copy of $(V,\omega_{\psi})$.

A {\bf Stein domain} is a regular
sub-level set $\{ \psi\leq c\}$; this is also called a
{\bf Stein filling} of the level set $\psi^{-1}(c)$ with
contact structure given by the $J$-invariant sub-bundle
of its tangent bundle.

If $\psi$ is also a Morse function,
then the index of any of its critical points is at most equal to
$(\dim_{\R}\!V)/2$, cf.\ Remark~\ref{rem:Morse} below. A Stein manifold
or domain is called {\bf subcritical} if $\psi$ is Morse
with all critical points of Morse index strictly smaller than
$(\dim_{\R}\!V)/2$.
\end{defn}

\begin{rem}
Stein fillability of a contact manifold $(M,\xi)$ can also be defined
by requiring the existence of a compact complex manifold $V$ with
boundary~$M$, admitting a strictly plurisubharmonic function $\psi$ for which
$M$ is a regular level set $\psi^{-1}(c)$. The interior $\Int (V)$
then admits an exhausting strictly plurisubharmonic function
(with the same level sets as~$\psi$): simply replace $\psi$ by
$h\circ\psi$ with $h\co(-\infty,c)\rightarrow\R$ convex and strictly
increasing, with $h(x)\rightarrow\infty$ as $x\rightarrow c$. By Grauert's
famous theorem~\cite{grau58}, $\Int (V)$ is then again a Stein manifold.
Moreover, thanks to Gray stability, a level set sufficiently close
to $\infty$ will be a contactomorphic copy of~$(M,\xi)$.
\end{rem}

\begin{cor}
\label{cor:Stein}
Any smooth compact hypersurface of contact type (with respect to
some symplectic form~$\omega_{\psi}$) in a subcritical Stein
manifold satisfies the strong Weinstein conjecture. If the hypersurface is
of restricted contact type, every contact form defining the
induced contact structure has a contractible periodic Reeb orbit.
\end{cor}

\begin{proof}
By a result of Cieliebak~\cite{ciel02}, cf.~\cite{ciel},
any subcritical Stein manifold is symplectomorphic to
a split one $(V\times\C, J_V\oplus i)$, where $(V,J_V)$
is some Stein manifold. So we have a strictly plurisubharmonic function
$\psi_V$ on $V$, and $\psi:=\psi_V+|z|^2/4$ is a strictly
plurisubharmonic function on $V\times\C$.

Thus, we are dealing with a compact hypersurface $(M,\xi)$
of (restricted) contact type in $(V\times\C,\omega_{\psi})$. Without loss of
generality we may take $M$ to be connected.
Since $V\times\C$ has trivial homology in codimension~$1$,
$M$ separates $V\times\C$
into a bounded and an unbounded part. Choose a regular level
set $\psi^{-1}(c)$ containing $M$ in the interior, and write
$W$ for the part between $M$ and $\psi^{-1}(c)$.
The Liouville vector field $Y$ near $M$ (or on all of~$W$ in
the case of restricted contact type) points into~$W$ along~$M$,
otherwise Theorem~\ref{thm:McDuff} would be violated.
Hence Corollary~\ref{cor:main}
applies to this symplectic cobordism~$W$. (For this qualitative result,
it is irrelevant that we have replaced $B_1$ by~$\C$.)
\end{proof}

As this proof shows, the corollary is close in spirit to
the work of Floer et al.~\cite{fhv90}, where the
existence of closed characteristics is proved in split symplectic
manifolds $P\times\C^l$ with $P$ closed and $\pi_2(P)=0$.
 
The following corollary extends earlier work of
Andenmatten~\cite[Theorem~1.4]{ande99} and Yau~\cite{yau04}, who
impose additional homological conditions on the Stein filling.

\begin{cor}
\label{cor:filling}
If $(M,\xi)$ is a subcritically Stein fillable contact manifold,
any contact form defining $\xi$ has a contractible
periodic Reeb orbit.
\end{cor}

\begin{proof}
By assumption, $(M,\xi)$ is a level set $\psi^{-1}(c)$ of an
exhausting strictly plurisubharmonic Morse function $\psi$ on
a Stein manifold $(V,J)$. A contact form defining $\xi$,
the $J$-invariant sub-bundle of~$TM$, is given by the restriction of
the global primitive $-\rmd\psi\circ J$ of $\omega_{\psi}$.
So we are in the restricted contact type case of the preceding corollary.
\end{proof}

\begin{rem}
The Floer-homological methods of Viterbo~\cite{vite99} and
Frauenfelder--Schlenk~\cite{frsc07} produce a closed Reeb orbit
contractible in the ambient manifold. In the situation of
Corollary~\ref{cor:filling}, $M$ is disjoint from the
isotropic skeleton of the Stein filling, and by general position
a closed orbit contractible in the subcritical filling is also contractible
in $M$ itself. So that last corollary can alternatively be derived from
their result.
\end{rem}
\subsection{Capacities and non-squeezing}
\label{subsection:capacities}
Given a closed manifold $M$ with contact form $\alpha$ we write
$\inf (\alpha)$ for the infimum of all positive periods
of closed orbits of the Reeb vector field~$R_{\alpha}$. When there
are no closed Reeb orbits, we have $\inf (\alpha)=\infty$,
otherwise an Arzel\`a--Ascoli type argument as in \cite[p.~109]{hoze94} shows
that $\inf (\alpha)$ is a minimum, and in particular positive;
the latter is also a simple consequence of the flow box theorem.

Let $(V,\omega)$ be a $2n$-dimensional symplectic manifold. The manifold
$V$ is allowed to be non-compact or disconnected. It may also
have non-empty boundary, in which case $V$ should be replaced
by $\Int (V)$ in the following definition of a symplectic
invariant of $(V,\omega)$:
\[ c(V,\omega):=\sup_{(M,\alpha)}\{\inf (\alpha)|\,
\text{$\exists$ contact type embedding
$(M,\alpha)\hookrightarrow (V,\omega)$}\} .\]
Here the supremum is taken over all closed, but not necessarily
connected, contact manifolds $(M,\alpha)$ of dimension $2n-1$.
By a contact type
embedding $j\co (M,\alpha)\hookrightarrow (V,\omega)$ we mean that there
is a Liouville vector field $Y$ for $\omega$ defined near $j(M)$
such that $j^*(i_Y\omega)=\alpha$.

In $\R^{2n}\equiv\C^n$ with its standard symplectic form
$\omega_{\st}=(i/2)\, \rmd\bfz\wedge \rmd\overline{\bfz}$
write $B^{2n}_r$ for the open $2n$-ball of radius $r$ and
$Z_r=\C^{n-1}\times B^2_r$ for the cylinder over
the open $2$-ball of radius~$r$.
For $r=1$ we simply write $B$ and $Z$, respectively.

\begin{thm}
\label{thm:capacity}
The invariant $c(V,\omega)$ is a
symplectic capacity, i.e.\ it satisfies the following axioms:
\begin{description}
\item[Monotonicity] If there
exists a symplectic embedding $(V,\omega)\hookrightarrow (V',\omega')$,
then $c(V,\omega)\leq c(V',\omega')$.
\item[Conformality] For any $a\in\R^+$ we have
$c(V,a\omega)=a\, c(V,\omega)$.
\item[Normalisation] $c(B)=c(Z)=\pi$.
\end{description}
\end{thm}

\begin{proof}
Monotonicity and conformality are obvious from the definition.
The $(2n-1)$-sphere of radius $r<1$ with its standard contact form
has all Reeb orbits closed of period $\pi r^2$, and it admits
a contact type embedding into both $B$ and~$Z$; cf.~\cite{geze},
where this is computed explicitly in the $4$-dimensional case.
This implies that $c(B)$ and $c(Z)$ are bounded from below by~$\pi$.

If $j$ is a contact type embedding of some $(2n-1)$-dimensional
contact manifold $(M,\alpha)$ into $B$ or~$Z$,
then the image $j(M)$ is contained in the interior of an ellipsoid
\[ E=E_{\varepsilon}(b,\dots ,b,1)
    :=\Bigl\{ \sum_{k=1}^{n-1}\frac{|z_k|^2}{b^2}+
    |z_n|^2\leq 1-\varepsilon\Bigr\} \]
for $b>0$ sufficiently large and $\varepsilon>0$
sufficiently small. Now our main theorem applies
to the symplectic cobordism given by the region in $(Z,\omega_{\st})$ between
$j(M)$ and~$\partial E$.
\end{proof}

More generally, this argument shows that if $(P,\omega_P)$ is
a connected but not necessarily compact $(2n-2)$-dimensional
symplectic manifold admitting a plurisubharmonic potential $\psi_P$
(with the boundary as a regular level set in the compact case),
then $c(P\times B_1,\omega_P+\rmd x\wedge\rmd y)\leq\pi$.
It is not difficult to give examples where equality holds, e.g.\
if $P$ is a cotangent bundle or a Stein manifold of finite type.

As an immediate consequence of Theorem~\ref{thm:capacity},
we have Gromov's non-squeez\-ing theorem~\cite[p.~310]{grom85}:

\begin{cor}[Gromov]
There is a symplectic embedding $B^{2n}_r\hookrightarrow Z_R$
if and only if $r\leq R$.\hfill\qed
\end{cor}

Our main theorem allows us to define other capacities in
a similar fashion. One option is to take the infimum over
the total action of null-homologous Reeb links. A further
possibility, for exact symplectic manifolds, is to
work with restricted contact type embeddings and then to take
the infimum over contractible Reeb orbits, see~\cite{geze}.
A comprehensive survey on symplectic capacities can be found
in~\cite{chls07}.

In the case of exact symplectic manifolds $(V,\rmd\lambda)$,
there are in fact two sensible ways to introduce a capacity.
One, carried out in the $4$-dimensional setting in~\cite{geze},
is to consider restricted contact type embeddings for the
given primitive~$\lambda$, i.e.\ embeddings $j\co (M,\alpha)\rightarrow
(V,\rmd\lambda)$ with $j^*\lambda=\alpha$. This leads to an
invariant on the set of exact symplectic manifolds with given
primitive. Alternatively, and more in the
spirit of an Ekeland--Hofer capacity~\cite{ekho88},
one can define a capacity exclusively for subsets $U$ of a given
exact symplectic manifold $(V,\omega=\rmd\lambda)$.
This can be done via restricted contact type embeddings
$j\co (M,\alpha)\rightarrow (V,\omega)$, i.e.\ it is only required that
there be {\em some\/} global primitive $\lambda_j$ for $\omega$, depending on
the embedding~$j$, with $j^*\lambda_j=\alpha$, but in addition
the condition $j(M)\subset U$ is imposed. 

All these capacities have the same normalisation constants,
but we do not know if they are different, in general.
\subsection{Quantitative Reeb dynamics}
By way of an example, we show that our capacity can be used
to recover a result of Frauenfelder et al.~\cite[Remark~1.13.3]{fgs05}.
We improve the constant in their result by appealing to
classical geometry.

\begin{cor}
Let $(M,\alpha)\subset (\R^{2n},\omega_{\st})$ be
a compact hypersurface of contact type. Then
$\inf(\alpha)\leq(n/(2n+1))\pi(\diam(M))^2$.
\end{cor}

\begin{proof}
Since the symplectic form $\omega_{\st}$ is translation-invariant, we
have a contact type embedding of $(M,\alpha)$ into $B^{2n}_r$
for any $r$ greater than the circumradius of~$M$, which
by \cite{jung01} is
at most equal to $\sqrt{n/(2n+1)}\,\diam (M)$. This bound is optimal;
it is attained for the regular $2n$-simplex. Hence
\[ \inf (\alpha)\leq c(B^{2n}_r)=\pi r^2\;\;
\text{for any $r>\sqrt{n/(2n+1)}\,\diam (M)$.}\qed\]
\renewcommand{\qed}{}
\end{proof}

In the case of restricted contact type hypersurfaces, one
obtains the same quantitative estimate on {\em contractible\/}
closed Reeb orbits, cf.~\cite{geze}.
\subsection{Cotangent bundles}
The Weinstein conjecture for contact type
hypersurfaces in a cotangent bundle $T^*L$
(with its canonical symplectic structure) is
of particular interest, since this includes the question
of closed characteristics on energy surfaces in classical
mechanical systems. The solution to the existence
question in this classical case is described in
\cite[Chapter~4.4]{hoze94}. Hofer--Viterbo~\cite{hovi88}
proved the existence of closed characteristics on contact type
hypersurfaces in $T^*L$ enclosing the zero section.
Viterbo~\cite[Theorem~3.1]{vite97},
\cite[p.~1020]{vite99} covers the case
where the fundamental group of $L$ is finite.

The following corollary provides new instances of the Weinstein
conjecture in cotangent bundles, and in fact gives the
strong version.

\begin{cor}
\label{cor:cotangent}
The strong Weinstein conjecture holds for closed contact type
hypersurfaces in $T^*(Q\times S^1)$, where $Q$ is
any closed manifold.
\end{cor}

\begin{proof}
Let $(M,\xi)\subset T^*(Q\times S^1)=T^*Q\times T^*S^1$
be a closed hypersurface of contact type. We want to show that
we can realise $(M,\xi)$ as a hypersurface of contact type
in a symplectic manifold of the form $P\times\C$ with $P$ as
in Section~\ref{subsection:S}. The cobordism $W$ will then be defined
by the part of $P\times\C$ between $M$ and a sufficiently high
level set $S$ of the strictly plurisubharmonic potential
on $P\times\C$. Then we choose $R$ large enough such that
$S\subset P\times B_R$ and apply Theorem~\ref{thm:main}. For the
qualitative statement of the corollary it is irrelevant that we need to
replace $B_1$ in the original formulation of the main theorem
by~$B_R$.

The canonical symplectic form on $T^*S^1\cong\R\times S^1$
is given by $\rmd s\wedge \rmd\theta$. The compact hypersurface $M$
lies in $T^*Q\times\{s>-a/2\}\subset T^*Q\times T^*S^1$ for $a>0$
large enough, and a symplectic embedding of $\{s>-a/2\}\subset T^*S^1$
into $(\C,\omega_{\st}=r\, \rmd r\wedge\rmd\theta)$ is given by
$(s,\theta)\mapsto \sqrt{2s+a}\,\rme^{i\theta}$.

In order to finish the proof, we need to equip $T^*Q$ with
an almost complex structure and an exhausting strictly plurisubharmonic
function. In \cite[Appendix~B]{nied06} it is
explained how this can be done, starting from
a Riemannian metric on~$Q$.
\end{proof}

\begin{rem}
Since the symplectic embedding $\{ s>-a/2\}\hookrightarrow\C$
is not surjective, the image of a hypersurface of restricted
contact type will only be of non-restricted contact type,
in general. So our argument does not allow us to make any statement
about contractible periodic Reeb orbits.
\end{rem}

Notice that a closed contact type hypersurface in a symplectic manifold
of the form $P\times\C$ is displaceable. In this situation,
a result of Frauenfelder--Schlenk~\cite[Theorem~3]{frsc07}
predicts the existence of a closed characteristic.
\subsection{Separating hypersurfaces}
In \cite{abw10}, Albers et al.\ collect conditions on
a $3$-dimensional contact manifold that prevent the
existence of non-separating contact type embeddings into
any $4$-dimensional symplectic manifold. Our main theorem
and its consequences for the quantitative Reeb dynamics
allow us to make such statements in higher dimensions, e.g.\
for spheres and ellipsoids. However, as we need to control the
Reeb dynamics, we need to fix the induced contact form on the
hypersurface.

Here is a simple example, pointed out to us by Max
D\"orner. Consider the standard Liouville form $\lambda_{\st}
=(\bfx\,\rmd\bfy -\bfy\,\rmd\bfx)/2$ on~$\R^{2n}$, and write
$\alpha_r$ for its restriction to
the tangent bundle of the sphere $S_r:=S^{2n-1}_r$ of radius~$r$.

Recall from Section~\ref{subsection:capacities} that when
we speak of a contact type embedding of $(S_r,\alpha_r)$
into a symplectic manifold $(V,\omega)$,
the contact form $\alpha_r$ being given {\em a priori}, we mean
that there is a Liouville vector field $Y$ for $\omega$
defined near and transverse to $S_r\subset V$ with
$i_Y\omega|_{TS_r}=\alpha_r$.

\begin{prop}
Any contact type embedding of $(S_r,\alpha_r)$
into a closed $\kappa$-semi\-positive symplectic manifold
$(V,\omega)$ with $\pi r^2\leq\kappa$ is separating.
\end{prop}

\begin{proof}
Suppose we have a non-separating contact type embedding of
$(S_r,\alpha_r)$ into $(V,\omega)$.
Then a neighbourhood of $S_r\subset (V,\omega)$
looks like a neighbourhood of the sphere of radius $r$
in $(\R^{2n},\omega_{\st}=\rmd\lambda_{\st})$.
Remove an open tubular neighbourhood around $S_r\subset V$
corresponding to the shell between $S_{r-\varepsilon}$
and $S_{r+\varepsilon}$ in the euclidean model, where
$\varepsilon >0$ has been chosen sufficiently small.
This defines a symplectic cobordism
$(W,\omega)$ from $(S_{r+\varepsilon},\alpha_{r+\varepsilon})$
to $(S_{r-\varepsilon},\alpha_{r-\varepsilon})$
satisfying the assumptions of Theorem~\ref{thm:main},
where in (C1) we replace $\pi$-semipositivity by
$\pi r^2$-semipositivity. This tells us that there should be a Reeb link in
$(S_{r+\varepsilon},\alpha_{r+\varepsilon})$
of total action less than~$\pi r^2$. But all the simple Reeb orbits
of $(S_r,\alpha_r)$ are closed of period $\pi r^2$,
cf.~\cite{geze}. This contradiction proves the proposition.
\end{proof}

\begin{rem}
The same argument applies to any ellipsoid whose minimal
half-axis satisfies the inequality in the proposition.
\end{rem}
\section{Completing the symplectic cobordism} 
\label{section:completing}
We now begin with the preparations for the proof of the main theorem.
We define a `completion' of our symplectic cobordism
$W$ which contains holomorphic spheres, and we describe
some simple properties of these holomorphic spheres.
\subsection{The symplectic cap}
\label{subsection:cap}
The initial step in the proof of Theorem~\ref{thm:main} is
analogous to the arguments of McDuff in~\cite{mcdu91}.
We complete $W$ by attaching a negative half-symplectisation along~$M$,
as in Section~\ref{subsection:main}, and a (perforated) symplectic cap
$(C_{\infty},\omega_C)$ along~$S$, where $C_{\infty}$
is the closure of the component of $C\setminus S$ containing
$P_{\infty}$ (for the notation cf.\ Section~\ref{subsection:S}).

By gluing the convex boundary $(S,\alpha_S)$ of $(W,\omega)$
and the concave boundary $(S,\alpha_S)$ of $(C_{\infty},\omega_C)$
we obtain the symplectic manifold
\[ (\wtW,\tilde{\omega}):=
\bigl( (-\infty,0]\times M\cup_M W,\omega_-\bigr)\cup_S
(C_{\infty},\omega_C);\]
see Figure~\ref{figure:open-cap}.

\begin{figure}[h]
\labellist
\small\hair 2pt
\pinlabel $W$ at 388 231
\pinlabel $C_{\infty}$ at 690 425
\pinlabel $M$ at 116 237
\pinlabel $M_+$ [l] at 663 90
\pinlabel $S$ [b] at 605 458
\pinlabel $P_{\infty}$ [r] at 747 359
\pinlabel ${\partial P\times\CP^1}$ [l] at 905 355
\endlabellist
\centering
\includegraphics[scale=0.32]{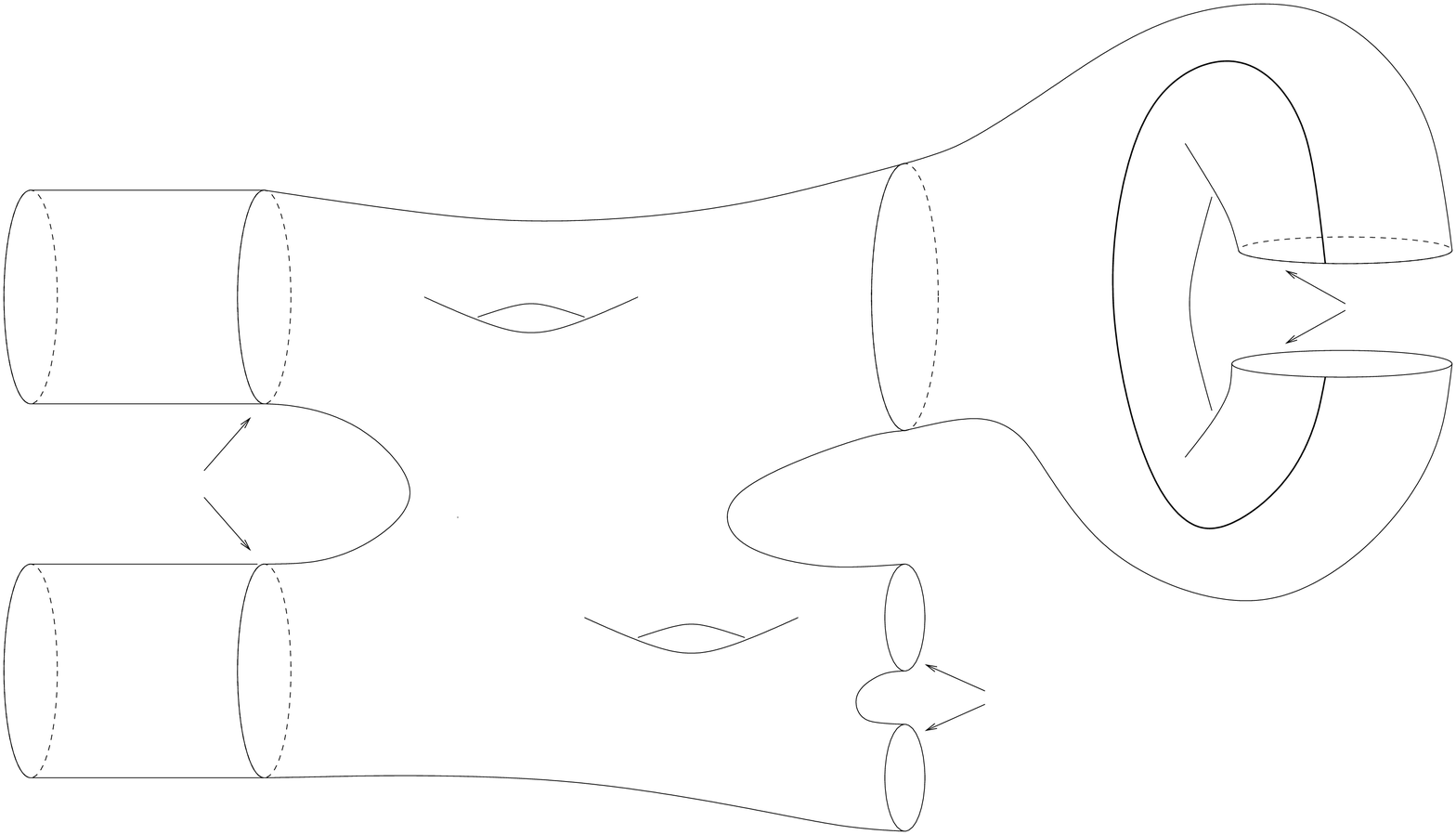}
  \caption{The symplectic manifold $(\wtW,\tilde{\omega})$.}
  \label{figure:open-cap}
\end{figure}

\begin{rem}
\label{rem:Morse}
Notice that the Morse index of a non-degenerate critical point
of a plurisubharmonic function is at most half the dimension
of the almost complex manifold. Otherwise the negative
definite subspace of the Hessian at the critical point would contain
a complex line, and by \cite{niwo63} we could then find
a local holomorphic curve tangent to that line. Such a curve would
violate the maximum principle. It follows that if $\psi_P$
is a Morse function, then the $(2n-2)$-dimensional manifold $P$
has the homotopy type of a complex of dimension at most $n-1$.
The homology exact sequence of the pair $(P,\partial P)$
then shows that $\partial P$ is connected when $n\geq 3$.

By a $C^2$-small perturbation compactly supported in $\Int (P)$
we can turn any given $\psi_P$ into a Morse strictly plurisubharmonic
function, so this topological conclusion about $\partial P$
always holds in our set-up. The apparently disconnected
$\partial P$ in Figure~\ref{figure:open-cap} is an artefact
of the lack of dimensions.

The same argument applied to the connected components
of $\psi^{-1}([0,c])$ shows that all these components
have a connected boundary. This means that each component
of the level set $\psi^{-1}(c)$, and hence in particular~$S$
(which is a collection of such components) is separating.
\end{rem}

For more about plurisubharmonic functions on almost
complex manifolds see~\cite{ciel}.
\subsection{The almost complex structure on $\wtW$}
On the symplectic manifold $(\wtW,\tilde{\omega})$ we choose an almost
complex structure $J$ tamed by~$\tilde{\omega}$, subject to the
following conditions:
\begin{itemize}
\item[(J1)] On $C_{\infty}\subset C$, the almost complex structure $J$
equals the split structure $J_P\oplus i$.
\item[(J2)] On the cylindrical end $(-\infty,0]\times M$,
the almost complex structure $J$ is cylindrical
and symmetric in the sense of \cite[p.~802, 807]{behwz03}, i.e.\ it
preserves $\xi=\ker\alpha$ and satisfies $J\partial_s=R_{\alpha}$.
\item[(J3)] On a neighbourhood of~$M_+$, the almost complex structure $J$
equals~$J_+$ (cf.~(C4)).
\item[(J4)] Outside the regions described in (J1) to (J3),
the almost complex structure is chosen so as to satisfy certain
genericity assumptions that will be described later.
\end{itemize}
\subsection{Holomorphic spheres in $(\wtW,J)$}
Before defining and studying the moduli space of holomorphic spheres
in $(\wtW,J)$ representing a certain homology class, we collect
some information about more general holomorphic spheres that will
be relevant in the bubbling-off analysis.

The $J$-convexity of $M_+$
allows one to write $M_+$ as a level set of a strictly plurisubharmonic
function defined in some collar neighbourhood $U_+$
of~$M_+$, cf.\ \cite[Remark~4.3]{geze}. Then the maximum principle
holds for $J$-holomorphic curves in~$U_+$.

Recall from Section~\ref{subsection:S} that $S$ is a collection of
components of a regular level set $\psi^{-1}(c)$.
Define a closed neighbourhood of $\partial P\times\CP^1$ by
\[ U_{\partial}:=\{ p\in P\co \psi_P (p)\geq c\} \times \CP^1
\subset C_{\infty} .\]
This neighbourhood is obviously foliated by holomorphic spheres
$\{ p\}\times\CP^1$.

\begin{lem}
\label{lem:spheres}
Let $u\co\CP^1\rightarrow\wtW$ be a smooth non-constant $J$-holomorphic
sphere.
\begin{itemize}
\item[(i)] If $u(\CP^1)\cap C_{\infty}\neq\emptyset$, then
$u(\CP^1)\cap P_{\infty}\neq\emptyset$.
\item[(ii)] If $u(\CP^1)\cap U_{\partial}\neq\emptyset$,
then $u(\CP^1)\subset U_{\partial}$ and $u$ is of the form
$z\mapsto (p,v(z))$ with some holomorphic
branched covering $v$ of $\CP^1$ by itself.
\item[(iii)] If $u(\CP^1)\subset C_{\infty}$, then $u$ is one of
the spheres in {\rm (ii)}.
\item[(iv)] $u(\CP^1)\cap U_+=\emptyset$.
\end{itemize}
\end{lem}

\begin{proof}
(i) Suppose $u$ is a holomorphic sphere intersecting $C_{\infty}$
but not $P_{\infty}$. Then the strictly plurisubharmonic function $\psi$
is defined on $u(\CP^1)\cap C_{\infty}$ and attains a maximum
in the interior, forcing $u(\CP^1)\cap C_{\infty}$ to lie
in a level set of~$\psi$. So either $u(\CP^1)$ is completely
contained in the exact symplectic manifold $C_{\infty}\setminus P_{\infty}$,
or $u$ is contained in $\wtW\setminus\Int (C_{\infty})$ and touches the
convex boundary $S$ of that manifold. Either alternative forces
$u$ to be constant.

(ii) On the preimage of $U_{\partial}$ we can write $u$ in the
form $u(z)=(u_1(z),u_2(z))\in P\times\CP^1$ with $u_1$
a $J_P$-holomorphic and $u_2$ a holomorphic function.
Thanks to the strictly plurisubharmonic
function $\psi_P$ on~$P$, the maximum principle applies to~$u_1$,
so as in (i) we argue that $u$ must be globally of the
form $z\mapsto (p,u_2(z))$. The non-constant holomorphic
map $u_2$ is a branched covering $\CP^1\rightarrow\CP^1$.

(iii) A sphere $u$ with image contained in $C_{\infty}$ can be
globally written as in (ii), that is, $u(z)=(p,u_2(z))$.
If $\psi_P(p)\geq c$, then this is one of the spheres from~(ii).
If $\psi_P(p)<c$, the point $(p,0)\in P\times B_1\subset P\times\CP^1$
does not lie in $C_{\infty}$, so $u_2$ is not surjective, and hence
constant.

(iv) If $u$ is a holomorphic sphere intersecting~$U_+$, then
with the strictly plurisubharmonic function defined on $U_+$
we argue as in (i) that $u$ must be constant.
\end{proof}

\begin{rem}
If $M_+$ is only weakly $J$-convex, which means that it is
the level set of a (non-strictly) plurisubharmonic function,
then the maximum principle still applies, but
there can be non-constant holomorphic spheres entirely contained
in a level set. However, such spheres cannot occur in
a bubble tree arising as the limit of spheres in the
moduli space considered in the next section, because one
sphere in such a bubble tree, as we shall see,
always intersects~$P_{\infty}$, so at least one sphere
would touch a level set but not be entirely contained in it.
So the results of the present paper remain valid under this weaker
assumption on~$M_+$.
\end{rem}
\section{The moduli space of holomorphic spheres}
\label{section:moduli}
In this section we define the relevant moduli spaces of holomorphic
spheres in the completed symplectic cobordism. These moduli spaces
will be shown to be non-compact, which then leads to a
proof of Theorems~\ref{thm:main} and~\ref{thm:McDuff}.
\subsection{Spheres in a fixed homology class}
Fix a point $*\in\partial P$. Then
\[ F:= \{ *\}\times\CP^1\subset\wtW \]
is a holomorphic sphere in $(\wtW,J)$.
Write $\wtMM$ for the moduli space of smooth $J$-holomorphic spheres
$u\co \CP^1\ra\wtW$ that represent the class $[F]\in H_2(\wtW)$.

The intersection number
of the classes $[F]\in H_2(\wtW)$ and $[P_{\infty}]\in
H_{2n-2}(\wtW,\partial\wtW)$ equals~$1$, so with Lemma~\ref{lem:spheres}
we see that any holomorphic sphere in the class $[F]$ that touches
$U_{\partial}$  or is completely contained in $C_{\infty}$
must be of the form $z\mapsto (p,\phi(z))$
with $\psi_P(p)\geq c$ and $\phi$ an automorphism of $\CP^1$.

\begin{prop}
For a generic choice of $J$, the moduli space $\wtMM$ is a
smooth manifold (with boundary) of dimension $2n+4$.
\end{prop}

\begin{proof}
The observation before the proposition tells us that near
its boundary the moduli space $\wtMM$ looks like a neighbourhood
of $\partial P\subset P$ crossed with the $6$-dimensional automorphism
group $\Aut (\CP^1)$, i.e.\ like a manifold with boundary of the
claimed dimension. So for all practical purposes we can
apply transversality arguments as if $\wtMM$ had no boundary.
A further consequence of the homological intersection of
$[F]$ and $[P_{\infty}]$ being equal to $1$ is that all spheres in the
class $[F]$ will be simple (i.e.\ not multiply covered).

The moduli space $\wtMM$ will be a manifold provided we can choose
$J$ to be regular in the sense of~\cite[Definition~3.1.4]{mcsa04},
i.e.\ such that the linearised Cauchy--Riemann operator $D_u$
is surjective for each $u\in\wtMM$.

As to (J1), regularity for spheres contained in $C_{\infty}$
follows from their explicit description given
before the proposition, cf.\ \cite[Corollary~3.3.5]{mcsa04}.

In the regions where the choice of
$J$ is prescribed by conditions (J2) and (J3), the maximum principle
applies, so no non-constant holomorphic sphere can lie entirely in one
of these regions. By \cite[Remark~3.2.3]{mcsa04},
the freedom of choosing $J$ in the complementary region then
suffices to achieve regularity for all holomorphic spheres
in~$\wtW$, cf.\ \cite[Remark~4.1.(2)]{geze}.

By \cite[Theorem~3.1.5]{mcsa04}, the dimension of the manifold
$\wtMM$ is given by $2n+2c_1([F])$. Since the normal bundle
of $F$ in the product manifold $P\times\CP^1$ is trivial (as a complex
bundle), $c_1([F])$ equals the Euler characteristic of the
sphere~$F$.
\end{proof}

The quotient space 
\[ \MM :=\wtMM\times_{\Aut (\CP^1)} \CP^1\]
of $\wtMM\times\CP^1$ under the diagonal
action of the $6$-dimensional automorphism group
$\Aut (\CP^1)=\mathrm{PGL}(2,\C )$, where
$\phi\in\Aut (\CP^1)$ acts by
\[ (u,z)\longmapsto (u\circ\phi^{-1},\phi(z)),\]
is then a $2n$-dimensional manifold, since the spheres in
$\wtMM$ being simple implies that this action is free.
Furthermore, there is a well-defined evaluation map
\[ \begin{array}{rccc}
\ev\co & \MM   & \lra        & \wtW\\
       & [u,z] & \longmapsto &  u(z),
\end{array}\]
where $[u,z]$ denotes the class represented by $(u,z)$.
\subsection{Spheres intersecting an arc}
\label{subsection:arc}
Let $\gamma$ be a proper embedding of the interval $[0,1]$
or $[0,1)$ into $\wtW$, with $\gamma(0)\in F$ and no other
image point of $\gamma$ in $\partial P\times\CP^1$.
So in the case of the closed interval $[0,1]$ we must have
$\gamma(1)\in M_+$; in the
case of the half-open interval $[0,1)$ we go to $-\infty$
in the cylindrical end $(-\infty ,0]\times M$ as we approach~$1$.
Set
\[ \MM_{\gamma}:=\ev^{-1}(\gamma), \]
where by slight abuse of notation we identify $\gamma$
with its image in~$\wtW$.

\begin{prop}
\label{prop:Mgamma}
Given $\gamma$, a generic choice of~$J$ can be made such that
$\MM_{\gamma}$ is a $1$-dimensional manifold including one component
diffeomorphic to a half-open interval. In particular,
$\MM_{\gamma}$ is not compact.
\end{prop}

\begin{proof}
Under the identification of $F=\{ *\}\times\CP^1$ with $\CP^1$,
the space $\MM_{\gamma}$ contains the class $[\id_{\CP^1},\gamma (0)]$,
so $\MM_{\gamma}$ is non-empty.

For $\MM_{\gamma}$ to be a $1$-dimensional
manifold (away from potential boundary points), we
need to ensure that the evaluation map $\ev$ be transverse
to the submanifold $\gamma$ of~$\wtW$.
By \cite[Theorem~3.4.1 and Remark~3.4.8]{mcsa04},
for generic $J$ transversality holds for all simple spheres not
contained entirely in a region where $J$ has been fixed by
one of the conditions (J1) to (J3). This is the generic choice
we want to make in (J4). The aforementioned theorem from \cite{mcsa04}
only treats the case without boundary, but it can still
be applied here. Indeed, for the interval $[0,1]$ we have $\gamma(1)\in M_+$,
where by Lemma~\ref{lem:spheres}~(iv) there are no non-constant holomorphic
spheres. A neighbourhood of the boundary point $\gamma(0)$ in $\gamma$
lies in~$U_{\partial}$, where the explicit description of holomorphic spheres
in the class $[F]$ as maps of the form $z\mapsto (p,\phi (z))$
with $\phi\in\Aut (\CP^1)$ shows that at any point
$[u,z]\in\MM$ with $u(z)\in U_{\partial}$ the evaluation map
$\ev\co\MM\ra\wtW$ is submersive.

For $(p,w)\in U_{\partial}\subset P\times\CP^1$,
the preimage $\ev^{-1}(p,w)$ must be of the form
$[(p,\phi),\phi^{-1}(w)]$, where $(p,\phi)$ denotes the holomorphic sphere
$z\mapsto (p,\phi (z))$, so this preimage consists in fact of the
single point $[(p,\id_{\CP^1}),w]$.

So near the preimage of $\gamma(0)$, the moduli space $\MM_{\gamma}$
looks like a single half-open interval, and there are no other boundary
points in $\MM_{\gamma}$. We conclude that the corresponding component is a
half-open interval.
\end{proof}
\subsection{Stable maps}
\label{subsection:stable}
By Proposition~\ref{prop:Mgamma} we can find a sequence
in $\MM_{\gamma}$ without any convergent subsequence.
We now want to show that such a sequence
cannot have any subsequence Gromov-converging to
a stable map in the sense of \cite[Definition~5.1.1]{mcsa04}.
Together with the compactness theorem from symplectic field
theory~\cite{behwz03}
this will imply, for a generic choice of contact form~$\alpha$,
that there has to be a Gromov--Hofer-convergent
subsequence where breaking occurs. This will lead to the existence of
periodic Reeb orbits.

Naively speaking, the non-existence of a Gromov-convergent
subsequence follows from the fact that bubbling is a phenomenon in
codimension at least~$2$ (in the $\pi$-semipositive situation)
--- see the dimension formula for the moduli space
$\MM^*_{T'}(\{ A_{\beta}\})$ in the proof of Lemma~\ref{lem:tree}
below --- and we are only considering a $1$-parameter family
of holomorphic curves.

For the formal argument, suppose $[u_{\nu},z_{\nu}]$ is a sequence in
$\MM_{\gamma}$ with $(u_{\nu},z_{\nu})$ Gromov-convergent to a stable map
$(\{ u_{\alpha}\}_{\alpha\in T},z)$ modelled on a tree~$T$.
Write $e(T)$ for the number of edges of~$T$. We want to show $e(T)=0$,
in which case the limit would be a classical one in the
$C^{\infty}$-topology.

By \cite[Proposition~6.1.2]{mcsa04} we find a simple stable map
$(\{ v_{\beta}\}_{\beta\in T'},z')$, i.e.\ with each
non-constant $v_{\beta}$ a simple map and with
different non-constant spheres having distinct images, such that
\[ \bigcup_{\beta\in T'}v_{\beta}(\CP^1)=
\bigcup_{\alpha\in T}u_{\alpha}(\CP^1). \]
Moreover, with $A_{\beta}\in H_2(\wtW)$ denoting the homology class
represented by the holomorphic sphere $v_{\beta}$,
and with suitable weights $m_{\beta}\in\N$, we have
\[ [F]=\sum_{\beta\in T'}m_{\beta}A_{\beta}.\]

There are two distinguished vertices in the bubble tree $T'$;
these may coincide. One is the bubble $v_{\beta_0}$
corresponding to the limit marked point~$z'$, in particular
$v_{\beta_0}(z')\in\gamma$. Notice that $v_{\beta_0}$
may well be constant, a so-called ghost bubble. This is
the case if the image of $z'$ on $\gamma$ happens to be the image
of a nodal point joining two non-constant bubbles.

The second is the bubble $v_{\beta_{\infty}}$ with
$v_{\beta_{\infty}}(\CP^1)\cap C_{\infty}\neq\emptyset$. We claim that
this property uniquely determines $v_{\beta_{\infty}}$, and
$m_{\beta_{\infty}}=1$. The homological intersection of
$F$ with $P_{\infty}$ equals~$1$, so there has to be at least one
non-constant sphere $v_{\beta}$ intersecting $P_{\infty}$. There are no
non-constant holomorphic spheres contained in $P_{\infty}$
(thanks to the strictly plurisubharmonic function
$\psi_P$ and the maximum principle), so positivity of intersection
with the holomorphic hypersurface $P_{\infty}$,
see~\cite[Proposition~7.1]{cimo07}, tells us that there is a unique
$v_{\beta}$ intersecting~$P_{\infty}$,
and this sphere has to be simple. From Lemma~\ref{lem:spheres}
it then follows that
\[ v_{\beta}(\CP^1)\subset\wtW\setminus C_{\infty}\;\;\mbox{\rm for}\;\;
\beta\neq\beta_{\infty}.\]

\begin{lem}
\label{lem:tree}
The tree $T'$ consists of a single vertex, i.e.\ $e(T')=0$.
\end{lem}

\begin{proof}
The symplectic energy $\int_{\CP^1}u^*\tilde\omega=\tilde\omega([u])$
of any sphere $u\in\wtMM$ equals $\omega_{\FS}([F])=\pi$.
Assuming $e(T')\geq 1$ (which implies that there
are at least two non-constant bubbles), we have
$0<\omega(A_{\beta})<\pi$ for every $\beta\in T'$
with $A_{\beta}\neq 0$.
For $\beta\neq\beta_{\infty}$ we may regard $A_{\beta}$
for homological computations
as a spherical class in $H_2(W)$. Our choice (J4) of almost
complex structure was made so as to guarantee regularity for spheres.
Then the moduli space of simple holomorphic spheres in any class
$A_{\beta}\neq 0$ is a non-empty manifold of dimension $2n+2c_1(A_{\beta})-6$,
which implies $c_1(A_{\beta})\geq 3-n$. So the requirement
that $(W,\omega)$ be $\pi$-semipositive implies
\[ c_1(A_{\beta})\geq 0\;\;\mbox{\rm for}\;\;
\beta\neq\beta_{\infty}.\]

As in the proof of Proposition~\ref{prop:Mgamma}, by possibly
refining the generic choice of almost complex structure in (J4)
we can ensure that the moduli space
$\MM^*_{T'}(\{ A_{\beta}\})$ of unparametrised simple stable maps
modelled on $T'$ (with a single marked point) is a smooth manifold.
As shown in \cite[Theorem~6.2.6]{mcsa04}, the dimension of this
moduli space equals
\[ \dim \MM^*_{T'}(\{ A_{\beta}\})=2n+2c_1\Bigl(\sum_{\beta}A_{\beta}\Bigr)
-4-2 e(T').\]

Now consider the evaluation map
\[ \begin{array}{rccc}
\ev_0\co & \MM^*_{T'}(\{ A_{\beta}\}) & \lra        & \wtW\\[1mm]
         & [\{v_{\beta}\},z']         & \longmapsto & v_{\beta_0}(z').
\end{array} \]
Since transversality of the evaluation map is automatic
at ghost spheres, we see as in Proposition~\ref{prop:Mgamma} that
$\ev_0^{-1}(\gamma)$ is a non-empty manifold of dimension
\begin{eqnarray*}
\dim\ev_0^{-1}(\gamma) & = & \dim \MM^*_{T'}(\{ A_{\beta}\})-(2n-1)\\
  & = & 2c_1\Bigl(\sum_{\beta}A_{\beta}\Bigr)-2 e(T')-3\\
  & = & 2c_1(A_{\beta_{\infty}})+
        2c_1\Bigl(\sum_{\beta\neq\beta_{\infty}}A_{\beta}\Bigr)
        -2 e(T')-3\\
  & \leq & 2c_1(A_{\beta_{\infty}})+
           2c_1\Bigl(\sum_{\beta\neq\beta_{\infty}}m_{\beta}A_{\beta}\Bigr)
           -2 e(T')-3\\
  & = & 2c_1([F])-2 e(T')-3\\
  & = & 1-2e(T').
\end{eqnarray*}
This dimension must be non-negative, hence
$e(T')=0$ after all.
\end{proof}

One of the reasons why a stable map may fail to be
simple is that it might contain several components having
the same image. In this case $e(T)$ would be larger than $e(T')$.
In our situation, however, an intersection argument with
the class $[P_{\infty}]$ as in the proof of
Proposition~\ref{prop:Mgamma} allows us to conclude that $T$,
likewise, must consist of a single vertex $\alpha_{\infty}$, with
$u_{\alpha_{\infty}}$ a simple map.
\subsection{Proof of the main theorem}
As shown in the final paragraph of~\cite{ach05}, by
invoking the Arzel\`a--Ascoli theorem one can reduce the proof
of Theorem~\ref{thm:main} to the case where in addition the contact
form $\alpha$ is assumed to be non-degenerate (i.e.\ where
the linearised Poincar\'e return map along closed orbits of the
Reeb vector field~$R_{\alpha}$, including multiples, does not have
an eigenvalue~$1$). Under this assumption, the compactness theorem
from symplectic field theory \cite[Theorem~10.2]{behwz03} applies.
This tells us that we can find a sequence in $\MM_{\gamma}$
convergent (in the sense of that paper) to a holomorphic building
of height $k_-|1$ with $k_->0$. The lowest level of this
building consists of holomorphic curves in the 
symplectisation $(\R\times M,\rmd(\rme^s\alpha))$ with positive punctures only
and symplectic energy smaller than~$\pi$, for the total symplectic energy
of the holomorphic building, which is a homological invariant,
equals~$\pi$. So these positive punctures constitute a null-homologous
Reeb link in $(M,\alpha)$ of total action smaller than~$\pi$.

Any holomorphic building of genus~$0$ and height $k_-|1$
with $k_->0$ contains at least two finite energy planes.
By the argument in the proof of Lemma~\ref{lem:spheres},
at most one of these can intersect~$P_{\infty}$;
any other finite energy plane in the top level
$W\cup_S C_{\infty}$ of the building stays inside~$W$.
In the Liouville case there can be no finite
energy plane in $W$ with a negative puncture, because by Stokes's
theorem its energy would be negative;
cf.~\cite[Lemma~5.16]{behwz03}. Hence, in the Liouville case
there has to be a finite energy plane with a positive
puncture in one of the lower levels $(\R\times M,\rmd(\rme^s\alpha))$
of the building, corresponding to a contractible
Reeb orbit of period smaller than~$\pi$.
This concludes the proof of Theorem~\ref{thm:main}.
\subsection{Proof of Theorem~\ref{thm:McDuff}}
\label{subsection:McDuff}
Arguing by contradiction, we assume that $M_+$ is not empty.
Then we can choose a path $\gamma$ in $W$ from $F$ to $M_+$
as in Section~\ref{subsection:arc}. Since the concave end $M$ of $W$
is assumed to be empty, the non-compactness of $\MM_{\gamma}$
can only be caused by bubbling, which is precluded by
the argument in Section~\ref{subsection:stable}. In other words,
$\MM_{\gamma}$ would have to be compact, contradicting
Proposition~\ref{prop:Mgamma}.
\subsection{Relation with the work of Oancea--Viterbo}
The recent paper of Oancea--Viterbo~\cite{oavi} also
grew from taking a fresh look at the work of McDuff~\cite{mcdu91}.
Their main interest is the topology of symplectic fillings.
Although their approach seems quite different on
a technical level, one can interpret their results
in our set-up as statements about the special case
where $M$ is empty (and hence so is $M_+$, by Theorem~\ref{thm:McDuff}).

In that special case, one can define the moduli space
$\MM_{-1,1,\infty}$ of holomorphic spheres $u\in\wtMM$
with $u(z)\in P\times\{ z\}$ for $z\in\{ -1,1,\infty\}$.
Requiring the image of three marked points to
lie on three respective hypersurfaces is essentially equivalent
to taking the quotient under the action of the automorphism
group $\Aut(\CP^1)$. For generic $J$ subject to conditions (J1) to (J4)
the moduli space $\MM_{-1,1,\infty}$
turns out to be a {\em compact\/} oriented manifold
(with boundary) of dimension
\[ 2n+2c_1([F])-3\cdot 2=2n-2.\]
Since the holomorphic spheres near the boundary $\partial P\times\CP^1$
of $\wtW$ are the obvious ones by Lemma~\ref{lem:spheres},
one sees that the evaluation map
\[ \ev\co\MM_{-1,1,\infty}\times\CP^1\rightarrow\wtW,\]
which maps boundary to boundary,
is of degree~$1$. The homological arguments in the proof of
\cite[Proposition~2.11]{oavi} go through unchanged,
even though in our set-up we work with manifolds with boundary. In this
way one can derive one of their main theorems~\cite[Theorem~2.6]{oavi}.
The homological conditions imposed there are necessary, in
general, to guarantee the compactness of the moduli space.
Let us point out one special case of their theorem
where these homological conditions are superfluous.

\begin{thm}[Oancea--Viterbo]
Let $S$ be as in Section~\ref{subsection:S} and $(W,\omega)$ a
symplectically aspherical strong symplectic filling of
$(S,\ker\alpha_S)$. Then the homomorphism
$H_j(S;\F)\rightarrow H_j(W;\F)$ on homology with coefficients
in any field~$\F$, induced by the inclusion $S\rightarrow W$,
is surjective in all degrees.\hfill\qed
\end{thm}
\begin{ack}
We are grateful to Alexandru Oancea for encouraging us with
this project. We thank Peter Albers, Fr\'ed\'eric Bourgeois,
Max D\"orner and Samuel Lisi for useful conversations,
and the referee for numerous perceptive comments and suggestions.
\end{ack}

\end{document}